\documentclass{amsart}
\usepackage{amsmath}
\usepackage{amsthm}
\usepackage{amsfonts,mathrsfs}
\usepackage{amssymb}
\usepackage{graphicx}
\usepackage{enumitem}
\usepackage{color}
\usepackage{verbatim}

\theoremstyle{plain}
\newtheorem{theorem}{Theorem}[section]
\newtheorem{lemma}[theorem]{Lemma}

\newtheorem{corollary}[theorem]{Corollary}
\newtheorem{proposition}[theorem]{Proposition}

\theoremstyle{definition}

\newtheorem{question}[theorem]{Question}

\newtheoremstyle{TheoremNum}
	{\topsep}{\topsep}              
  {\itshape}                      
  {}                              
  {\bfseries}                     
  {.}                             
  { }                             
  {\thmname{#1}\thmnote{ \bfseries #3}}
\newtheorem{remark}{Remark}

\newcommand{\F}{\mathbb F}
\newcommand{\K}{\mathbb K}

\newcommand{\cD}{\mathcal D}
\newcommand{\cA}{\mathcal A}
\newcommand{\cT}{\mathcal T}
\newcommand{\cG}{\mathcal G}
\newcommand{\cL}{\mathcal L}
\newcommand{\cH}{\mathcal H}
\newcommand{\cS}{\mathcal S}
\newcommand{\cM}{\mathcal M}
\newcommand{\cB}{{\mathcal B}}
\newcommand{\cX}{{\mathcal X}}
\newcommand{\cY}{{\mathcal Y}}

\newcommand{\cN}{\mathcal N}
\newcommand{\bu}{\mathbf u}

\newcommand{\cC}{\mathcal C}
\newcommand{\cU}{{\mathcal U}}
\newcommand{\cV}{{\mathcal V}}
\newcommand{\cW}{{\mathcal W}}

\newcommand{\Aut}{\mathrm{Aut}}

\newcommand{\GL}{\mathrm{GL}}

\newcommand{\PG}{\mathrm{PG}}
\newcommand{\Tr}{ \ensuremath{ \mathrm{Tr}}}
\newcommand{\N}{ \ensuremath{ \mathrm{N}}}

\newcommand{\lp}[2]{\mathscr{L}_{(#1,#2)}}
\newcommand{\RN}[1]{%
  \textup{\uppercase\expandafter{\romannumeral#1}}%
}
\newcommand{\rn}[1]{%
  \textup{\lowercase\expandafter{\romannumeral#1}}%
}

\newcommand{\la}{\langle}
\newcommand{\ra}{\rangle}

 \def\zhou#1 {\fbox {\footnote {\ }}\ \footnotetext { From Yue: {\color{red}#1}}}

\begin{document}
	\title[MRD codes with maximum idealisers]{MRD codes with maximum idealisers}
	\author[B.\ Csajb\'ok]{Bence Csajb\'ok\textsuperscript{\,1}}
	\thanks{
	The	research  was supported by the Italian National
	Group for Algebraic and Geometric Structures and their Applications (GNSAGA
	- INdAM) and by  the project ``VALERE: VAnviteLli pEr la RicErca" of the University of Campania ``Luigi Vanvitelli''.
	The first author is supported by the J\'anos Bolyai
	Research Scholarship of the Hungarian Academy of Sciences 
	and partially by OTKA Grant No. PD 132463 and OTKA Grant No. K 124950.
	The fourth author is supported by the National Natural Science Foundation of China (No.\ 11771451)}
	\address{\textsuperscript{1}MTA--ELTE Geometric and Algebraic Combinatorics Research Group\\
		ELTE E\"otv\"os Lor\'and University, Budapest, Hungary\\
		Department of Geometry\\
		1117 Budapest, P\'azm\'any P.\ stny.\ 1/C, Hungary}
	\email{csajbokb@cs.elte.hu}
	\author[G.\ Marino]{Giuseppe Marino\textsuperscript{\,2,\,3}}
	\address{\textsuperscript{2}Dipartimento di Matematica e Applicazioni ``Renato Caccioppoli"\\
		Università degli Studi di Napoli ``Federico II"\\
		Via Cintia, Monte S.Angelo I-80126 Napoli, Italy}
	\email{giuseppe.marino@unina.it}
	\address{\textsuperscript{3}Dipartimento di Matematica e Fisica\\
		Universit\`a degli Studi della Campania ``Luigi Vanvitelli''\\
		Viale Lincoln 5, I-\,81100 Caserta, Italy}
	\email{olga.polverino@unicampania.it}
	\author[O.\ Polverino]{Olga Polverino\textsuperscript{\,3}}
	\author[Y.\ Zhou]{Yue Zhou\textsuperscript{\,4}}
	\address{\textsuperscript{4}College of Liberal Arts and Sciences, National University of Defense Technology, 410073 Changsha, China}
	\email{yue.zhou.ovgu@gmail.com}
	
	\begin{abstract}
		Left and right idealisers are important invariants of linear rank-distance codes. In the case of maximum rank-distance (MRD for short) codes in $\F_q^{n\times n}$ the idealisers have been proved to be isomorphic to finite fields of size at most $q^n$. Up to now, the only known MRD codes with maximum left and right idealisers are generalized Gabidulin codes, which were first constructed in 1978 by Delsarte and later generalized by Kshevetskiy and Gabidulin in 2005. In this paper we classify MRD codes in $\F_q^{n\times n}$ for $n\leq 9$ with maximum left and right idealisers and connect them to Moore-type matrices. Apart from generalized Gabidulin codes, it turns out that there is a further family of rank-distance codes providing MRD ones with maximum idealisers for $n=7$, $q$ odd and for $n=8$, $q\equiv 1 \pmod 3$. These codes are not equivalent to any previously known MRD code. Moreover, we show that this family of rank-distance codes does not provide any further examples for $n\geq 9$.
	\end{abstract}
	\maketitle

\section{Introduction}\label{sec:intro}
For two positive integers $m$ and $n$ and for a field $\K$, let $\K^{m\times n}$ denote the set of all $m\times n$ matrices over $\K$. The \emph{rank metric} or the \emph{rank distance} on $\K^{m\times n}$ is defined by
\[d(A,B)=\mathrm{rank}(A-B),\]
for any $A,B\in \K^{m\times n}$.

A subset $\cC\subseteq \K^{m\times n}$ with respect to the rank metric is usually called a \emph{rank-metric code} or a \emph{rank-distance code}. When $\cC$ contains at least two elements, the \emph{minimum distance} of $\cC$ is given by
\[d(\cC)=\min_{A,B\in \cC, A\neq B} \{d(A,B)\}.\]
When $\cC$ is a $\K$-linear subspace of $\K^{m\times n}$, we say that $\cC$ is a $\K$-linear code and its dimension $\dim_{\K}(\cC)$ is defined to be the dimension of $\cC$ as a subspace over $\K$.

Let $\F_q$ denote the finite field of $q$ elements. For any $\cC\subseteq \F_q^{m\times n}$ with $d(\cC)=d$, it is well-known that
\[\#\cC\le q^{\max\{m,n\}(\min\{m,n\}-d+1)},\]
which is a Singleton like bound for the rank metric; see \cite{delsarte_bilinear_1978}. When equality holds, we call $\cC$ a \emph{maximum rank-distance} (MRD for short) code. More properties of MRD codes can be found in \cite{delsarte_bilinear_1978}, \cite{gabidulin_MRD_1985}, \cite{gadouleau_properties_2006}, \cite{morrison_equivalence_2014} and \cite{ravagnani_rank-metric_2015}.

Rank-metric codes, in particular MRD codes, have been studied since the 1970s and have seen much interest in recent years due to a wide range of applications including storage systems~\cite{roth_1991_Maximum}, cryptosystems~\cite{gabidulin_public-key_1995}, spacetime codes~\cite{lusina_maximum_2003} and random linear network coding~\cite{koetter_coding_2008}.

In finite geometry, there are several interesting structures, including quasifields, semifields, splitting dimensional dual hyperovals and maximum scattered subspaces,  which can be equivalently described as special types of rank-distance codes; see \cite{csajbok_maximum_2017}, \cite{dempwolff_dimensional_2014}, \cite{dempwolff_orthogonal_2015}, \cite{sheekey_new_2016}, \cite{taniguchi_unified_2014} and the references therein. In particular, a finite quasifield corresponds to an MRD code in $\F_q^{n\times n}$ of minimum distance $n$ and a finite semifield corresponds to an MRD code that is a subgroup of $\F_q^{n\times n}$ (see \cite{de_la_cruz_algebraic_2016} for the precise relationship). Many essentially different families of finite quasifields and semifields are known  \cite{johnson_handbook_2007}, \cite{lavrauw_semifields_2011}, which yield many inequivalent MRD codes in $\F_q^{n\times n}$ of minimum distance $n$.

There are several slightly different definitions of equivalence of rank-distance codes. In this paper, we use the following notion of equivalence.
	
	Two rank-distance codes $\cC_1$ and $\cC_2$ in $\K^{m\times n}$ are \emph{equivalent} if there exist $A\in\GL_m(\K)$, $B\in \GL_n(\K)$, $C\in\K^{m\times n}$ and $\rho\in\Aut(\K)$ such that
	\begin{equation}\label{eq:def_equiv}
	\cC_2=\{AM^{\rho}B+C \colon M \in\cC_1\}.
	\end{equation}
	
	The \emph{adjoint code} of a rank-metric code $\cC$ in $\K^{m\times n}$ is
	\[\cC^\top :=\{M^T\in \K^{n\times m}\colon M\in \cC  \},\]
	where $(\,.\,)^T$ denotes transposition. If $\cC$ is a linear MRD code then $\cC^\top$ is also a linear MRD code. For $m=n$, if $\cC_2$ is equivalent to $\cC_1$ or $\cC_1^\top$, then $\cC_1$ and $\cC_2$ are called \emph{isometrically equivalent}. An equivalence map from a rank-distance code $\cC$ to itself is also called an \emph{automorphism} of $\cC$.

When $\cC_1$ and $\cC_2$ are both additive and equivalent, it is not difficult to show that we can choose $C=0$ in \eqref{eq:def_equiv}.

In general, it is a difficult job to tell whether two given rank-distance codes are equivalent or not. There are several invariants which may help us distinguish them.
Given a $\K$-linear rank-distance code $\cC\subseteq \K^{m\times n}$, following \cite{liebhold_automorphism_2016} its \emph{left} and \emph{right idealisers} are defined as
\[L(\cC) =\{M\in\K^{m\times n} \colon MC\in \cC \text{ for all }C\in \cC  \},\]
and
\[R(\cC) =\{M\in\K^{m\times n} \colon CM\in \cC \text{ for all }C\in \cC  \},\]
respectively. The left and right idealisers can be viewed as a natural generalization of the middle and right nucleus of semifields \cite{lunardon_kernels_2017} and some authors call them in this way.
In general, we can also define the left nucleus of $\cC$ which is another invariant for semifields. However, for MRD codes in $\F_q^{m\times n}$ with minimum distance less than $\min \{m,n\}$, the left nucleus is always $\F_q$ which means that it is not a useful invariant; see \cite{lunardon_kernels_2017}.


The \emph{Delsarte dual code} of an $\F_q$-linear code $\cC\subseteq \F_q^{m\times n}$ is
\[\cC^\perp :=\{M\in \F_q^{m\times n} \colon \Tr(MN^T)=0 \text{ for all } N\in \cC  \}.\]
If $\cC$ is a linear MRD code then $\cC^\perp$ is also a linear MRD code as it was proved by Delsarte \cite{delsarte_bilinear_1978}.


Two MRD codes in $\F_q^{n\times n}$ with minimum distance $n$ are equivalent if and only if the corresponding semifields are isotopic \cite[Theorem 7]{lavrauw_semifields_2011}.
In contrast, it appears to be much more difficult to obtain inequivalent MRD codes in $\F_q^{n\times n}$ with minimum distance strictly less than $n$.
We divide the known constructions of inequivalent MRD codes in $\F_q^{n\times n}$ of minimum distance strictly less than $n$ into two types.
\begin{enumerate}
	\item The first type of constructions consists of MRD codes of minimum distance $d$ for arbitrary $2\le d\le n$.
	\begin{itemize}
		\item The first construction of MRD codes which was given by Delsarte~\cite{delsarte_bilinear_1978} and later rediscovered by Gabidulin~\cite{gabidulin_MRD_1985} and generalized by Kshevetskiy and Gabidulin~\cite{kshevetskiy_new_2005}. They are usually called the (generalized) Gabidulin codes. In 2016, Sheekey~\cite{sheekey_new_2016} found the so-called (generalized) twisted Gabidulin codes. They can be generalized into additive MRD codes \cite{otal_additive_2016}. Very recently, by using skew polynomial rings Sheekey~\cite{sheekey_new_arxiv} proved that they can be further generalized into a quite big family and all the MRD codes mentioned above can be obtained in this way.
		\item The non-additive family constructed by Otal and \"Ozbudak in \cite{otal_non-additive_2018}.
		\item The family appeared in \cite{trombetti_new_2017} which is related to the Hughes-Kleinfeld semifields.
	\end{itemize}
	\item The second type of constructions provides us MRD codes of minimum distance $d=n-1$.
	\begin{itemize}
		\item Non-linear MRD codes by Cossidente, the second author and Pavese \cite{cossidente_non-linear_2016} which were later generalized by Durante and Siciliano \cite{durante_nonlinear_MRD_2017}.
		\item Linear MRD codes associated with maximum scattered linear sets of $\PG(1,q^6)$ and $\PG(1,q^8)$ presented recently in \cite{BZZ,csajbok_newMRD_2017,csajbok_maximum_arxiv,MMZ,ZZ}.
	\end{itemize}
\end{enumerate}
For the relationship between MRD codes and other geometric objects such as linear sets and Segre varieties, we refer to \cite{lunardon_mrd-codes_2017}. For more results concerning maximum scattered linear sets and associated MRD codes, see \cite{bartoli_scattered_arxiv}, \cite{csajbok_classes_2017}, \cite{csajbok_maximum_2017}, \cite{csajbok_equivalence_2016}, \cite{csajbok_maximum_4_arxiv} and \cite{ShVdV}.

Compared to the known MRD codes in $\F_q^{n\times n}$ listed above, there are slightly more ways to get MRD codes in $\F_q^{m\times n}$ with $m<n$, see \cite{csajbok_maximum_2017},~\cite{donati_generalization_2017},~\cite{horlemann-trautmann_new_2015},~\cite{neri_genericity_2017} and \cite{schmidt_number_MRD_2017}.

For an MRD code $\cC$ in $\F_q^{n \times n}$, by \cite[Corollary 5.6]{lunardon_kernels_2017}, its left and right idealisers are isomorphic to finite fields of size at most  $q^n$. Moreover, according to \cite[Proposition 4.2]{lunardon_kernels_2017} if the left and right idealisers of an MRD code $\cC$ in $\F_q^{n \times n}$ are both isomorphic to $\F_{q^n}$, then the same holds for $\cC^\top$ and $\cC^\perp$.

Among the $\F_q$-linear MRD codes listed in (1) and (2), only the generalized Gabidulin codes have this special property. Thus, it is natural to ask whether there exist other MRD codes in  $\F_q^{n \times n}$ with maximum left and right idealisers. In this paper, we classify $\F_q$-linear MRD codes $\cC$ in $\F_q^{n\times n}$, $n\leq 9$, with $L(\cC)\cong R(\cC)\cong \F_{q^n}$ up to the adjoint and Delsarte dual operations. In particular, our classification includes new examples of such MRD codes for $n=7$, $q$ odd (cf.\ Theorem \ref{thm1} and Corollary \ref{cor1}), and for $n=8$, $q\equiv 1 \pmod 3$ (cf.\ Theorem \ref{thm2} and Corollary \ref{cor2}).

More precisely, we prove the following result.

\begin{theorem}\label{thm:main}
Let $\cC$ be an $\F_q$-linear MRD code in $\F_{q}^{n\times n}$ with left and right idealisers isomorphic to $\F_{q^n}$, $n\geq 2$.
\begin{itemize}
\item If $n\leq 6$ or $n=9$ then $\cC$ is equivalent to a generalized Gabidulin code.
\item If $n=7$ then  $\cC$ is equivalent to a generalized Gabidulin code or $q$ is odd and, up to the adjoint operation, $\cC$ is equivalent either to
\[\cC_7:=\{a_0X+a_1X^q+a_2X^{q^3} \colon a_0,a_1,a_2\in \F_{q^7}\}\] or to
\[\cC_7':=\{a_0X+a_1X^{q^3}+a_2X^{q^5}+a_3X^{q^6} \colon a_0,a_1,a_2,a_3\in \F_{q^7}\}.\]
\item If $n=8$ then $\cC$ is equivalent to a generalized Gabidulin code or $q\equiv 1 \pmod 3$ and, up to the adjoint operation, $\cC$ is equivalent either to
\[\cC_8:=\{a_0X+a_1X^q+a_2X^{q^3} \colon a_0,a_1,a_2\in \F_{q^8}\}\] or to
\[\cC_8':=\{a_0X+a_1X^{q^2}+a_2X^{q^3}+a_3X^{q^4}+a_4X^{q^5} \colon a_0,a_1,a_2,a_3,a_4\in \F_{q^8}\}.\]
\end{itemize}
(Note that $\cC_7'$ is equivalent to $\cC_7^\perp$ and $\cC_8'$ is equivalent to $\cC_8^\perp$.)
\end{theorem}

The rest of this paper is organized as follows: In Section \ref{sec:pre}, we prove several results concerning the representation and the equivalence of MRD codes with maximum left and right idealisers. Moreover, we also show connections between Moore matrices and such MRD codes. Section \ref{sec:constructions} includes the constructions and the classification results of Theorem \ref{thm:main}. In Section \ref{sec:nonexistence} we show a link between the Dickson-Guralnick-Zieve curves and a family of rank-metric codes in $\F_q^{n\times n}$, which provides the MRD codes of Section \ref{sec:constructions} for $n=7$ and $8$. By using some recent results on these curves, we can prove that the members of this family of rank-metric codes are not MRD for $n\geq 9$.

\section{Linearized polynomials and Moore matrices}
\label{sec:pre}
As we are working with rank-distance codes in $\F_q^{n\times n}$ in this paper, it is more convenient to describe codes in the language of \emph{$q$-polynomials} (or \emph{linearized polynomials}) over $\F_{q^n}$, considered modulo $X^{q^n}-X$. These polynomials are represented by the set
\[\lp{n}{q}[X]=\left\{\sum_{i=0}^{n-1} c_i X^{q^i}\colon c_i\in \F_{q^n} \right\}.\]
After fixing an ordered $\F_q$-basis $\{b_1,b_2,\ldots,b_n\}$ for $\F_{q^n}$ it is possible to give a bijection $\Phi$ which associates for each matrix $M\in \F_q^{n\times n}$ a unique $q$-polynomial $f_M\in \lp{n}{q}$. More precisely, put ${\bf b}=(b_1,b_2,\ldots,b_n)\in \F_{q^n}^n$, then $\Phi(M)=f_M$ where for each $\bu=(u_1,u_2,\ldots,u_n)\in \F_q^n$ we have $f_M({\bf b}\cdot \bu)={\bf b}\cdot \bu  M$. The trace map from $\F_{q^n}$ to $\F_q$ is defined by the $q$-polynomial \[\Tr_{q^n/q}(x)=x+x^q+\ldots+x^{q^{n-1}} \text{ for }x\in \F_{q^n}.\]

As we mentioned in the introduction, the most well-known family of MRD codes is called (generalized) Gabidulin codes. They can be described by the following subset of linearized polynomials:
\begin{equation}
\label{GGC}
\cG_{k,s} = \{a_0 x + a_1 x^{q^{s}} + \dots +a_{k-1} x^{q^{s(k-1)}}\colon a_0,a_1,\dots, a_{k-1}\in \F_{q^n} \},
\end{equation}
where $s$ is relatively prime to $n$. It is obvious that there are $q^{kn}$ polynomials in $\cG_{k,s}$. Each of them has at most $q^{k-1}$ roots (cf.\ \cite{Gow}) which means that this is an MRD code.

Given two rank-distance codes $\cC_1$ and $\cC_2$ which consist of linearized polynomials, they are equivalent if and only if there exist $\varphi_1$, $\varphi_2\in \lp{n}{q}[X]$ permuting $\F_{q^n}$, $\psi\in \lp{n}{q}[X]$ and $\rho\in \Aut(\F_q)$ such that
\[ \varphi_1\circ f^\rho \circ \varphi_2 + \psi\in \cC_2 \text{ for all }f\in \cC_1,\]
where $\circ$ stands for the composition of maps and $f^\rho(X)= \sum a_i^\rho X^{q^i}$ for $f(X)=\sum a_i X^{q^i}$.

For a rank-distance code $\cC$ given by a set of linearized polynomials, its left and right idealisers can be written as:
\[L(\cC)= \{ \varphi \in \lp{n}{q}\colon f\circ \varphi\in \cC \text{ for all }f\in \cC \},\]
\[R(\cC)= \{ \varphi \in \lp{n}{q}\colon \varphi \circ f\in \cC \text{ for all }f\in \cC \}.\]
Note that the left idealiser is written as $f\circ \varphi$ rather than $\varphi\circ f$ because of the definition of $\Phi$ and similarly for the right idealiser.

The idealisers of generalized twisted Gabidulin codes together with a complete answer to the equivalence between members in this family can be found in \cite{lunardon_generalized_2015}.

The \emph{adjoint} of a $q$-polynomial $f(x)=\sum_{i=0}^{n-1}a_i x^{q^i}$, with respect to the bilinear form $\langle x,y\rangle:=\Tr_{q^n/q}(xy)$, is given by
\[\hat{f}(x):=\sum_{i=0}^{n-1}a_{i}^{q^{n-i}} x^{q^{n-i}}.\]
If $\cC$ is a rank-metric code given by $q$-polynomials, then the \emph{adjoint code} $\cC^\top$ of $\cC$ is $\{\hat{f}\colon f\in\cC\}$.

In terms of linearized polynomials, the Delsarte dual can be interpreted in the following way \cite{sheekey_new_2016}:
\[\cC^\perp=\{f\colon b(f,g)=0 \text{ for all } g\in \cC \},\]
where $b\left( f,g \right)=\Tr_{q^n/q}\left(\sum_{i=0}^{n-1}a_ib_i\right)$ for $f(x)=\sum_{i=0}^{n-1}a_ix^{q^i}$ and $g(x)=\sum_{i=0}^{n-1}b_ix^{q^i}\in \F_{q^n}[x]$.

It is well-known and also not difficult to show directly that two linear rank-distance codes are equivalent if and only if their Delsarte duals or their adjoint codes are equivalent. This observation yields the following result which we will use without further mentioning throughout the paper.

\begin{proposition}
	Let $\cC$ and $\cC'$ be rank metric codes of $\F_{q}^{n\times n}$ such that $\cC$ is obtained from $\cC'$ via a finite combination (possibly with repetition) of the $\top$ and $\perp$ operations and the equivalence maps. Then $\cC$ is equivalent to a generalized Gabidulin code if and only if $\cC'$ is equivalent to a generalized Gabidulin code.
\end{proposition}
\begin{proof}
It follows from the fact that $\cG_{k,s}^\top$ is equivalent to $\cG_{k,s}$ and $\cG_{k,s}^\perp$ is equivalent to $\cG_{n-k,s}$.
\end{proof}

Usually, codes equivalent to those defined in \eqref{GGC} are also called generalized Gabidulin codes. Note that changing the basis $\{b_1,b_2,\ldots,b_n\}$ of $\F_{q^n}$ can alter the shape of the corresponding $q$-polynomials but provide equivalent codes. In this paper by a generalized Gabidulin code we always refer to a code defined exactly as in \eqref{GGC}. We decided along this notation since, as we will see, finding a nice shape of the representing $q$-polynomials has a crucial role in our investigation.

\subsection{Rank-distance codes with maximum nuclei}
First let us show that a rank-distance code in $\lp{n}{q}$ with maximum right and left idealisers has to be equivalent to a set of linearized polynomials in a special form.
\begin{theorem}
	\label{th:classification}
	Let $\cC$ be an $\F_q$-subspace of $\lp{n}{q}$. Assume that one of the left and right idealisers of $\cC$ is isomorphic to $\F_{q^n}$. Then there exists an integer $k$ such that $|\cC|=q^{kn}$ and $\cC$ is equivalent to
	\begin{equation}\label{eq:normalized_C_1}
		\cC = \left\{ \sum_{i=0}^{k-1}a_{i} X^{q^{t_i}}+ \sum^{n-1}_{j\notin\{t_0,t_1,\cdots, t_{k-1}\}}  g_j(a_0,\cdots,a_{k-1})X^{q^j}\colon a_0,\cdots,a_{k-1}\in \F_{q^n} \right\},
	\end{equation}
	where $0\le t_0<t_1<\cdots<t_{k-1}\le n-1$ and the $g_j$'s are $\F_q$-linear functions from $\F_{q^n}^k$ to $\F_{q^n}$. If the other idealiser of $\cC$ is also isomorphic to $\F_{q^n}$, then $\cC$ is equivalent to
	\begin{equation}\label{eq:normalized_C_2}
		\left\{ \sum_{i=0}^{k-1}a_{i} X^{q^{t_i}}\colon a_i\in \F_{q^n} \right\}.
	\end{equation}
\end{theorem}
\begin{proof}
	Let $\cN$ denote the idealiser of $\cC$ which is isomorphic to $\F_{q^n}$. All Singer cycles in $\GL(n,q)$ are conjugate, i.e.\ there exists an invertible $f\in \lp{n}{q}$ such that
	$\cN':=f \circ \cN \circ f^{-1}=\{aX\colon a\in \F_{q^n} \}$.
	It follows that when $\cN=R(\cC)$ then $R(\cC')=\cN'$ where $\cC'=f \circ \cC$, whereas when $\cN=L(\cC)$ then $L(\cC')=\cN'$ where $\cC'=\cC \circ f^{-1}$. It means that up to equivalence we may assume that
	\begin{equation}
	\label{eq:N=all}
		\cN=\{aX \colon a\in \F_{q^n} \}.
	\end{equation}
	If the other idealiser $\cM$ of $\cC$ is also isomorphic to $\F_{q^n}$, then by using another equivalence map we may also assume that $\cM=\cN$.
	
	First we prove \eqref{eq:normalized_C_1}. Let $t_0$ be an integer such that there exists $f_0(X)=\sum_{i=0}^{n-1}a_i X^{q^i}\in \cC$ with $a_{t_0}\neq 0$. If $\cN$ is the right idealiser of $\cC$, then, by \eqref{eq:N=all}, $\{af_0(X) \colon a \in \F_{q^n} \}\subseteq \cC$, which means that for any $a\in \F_{q^n}$ there is at least one polynomial in $\cC$ where the coefficient of $X^{q^{t_0}}$ equals $a$. If $\cN$ is the left idealiser of $\cC$, then, by \eqref{eq:N=all},  $\{f_0(aX) \colon a \in \F_{q^n} \}\subseteq \cC$. Again, it follows that for any $a\in \F_{q^n}$ there is at least one polynomial in $\cC$ in which the coefficient of $X^{q^{t_0}}$ equals $a$.
	
	If $|\cC|=q^n$, we have proved \eqref{eq:normalized_C_1}; otherwise there exist non-zero polynomials in $\cC$ where the coefficient of $X^{q^{t_0}}$ is $0$. Let us denote the set of all such polynomials by $\bar\cC$. It is easy to check that $\bar \cC$ is still an $\F_q$-subspace. Let $t_1\neq t_0$ be an integer such that there exists a polynomial $f_1(X)=\sum_{i=0}^{n-1}a_i X^{q^i}\in \bar \cC$ with $a_{t_1}\neq 0$. Again, if $\cN=R(\cC)$, by \eqref{eq:N=all}, we see that $\{af_1(X) \colon a \in \F_{q^n} \}\subseteq \bar \cC$, whence $\{a_0f_0(X) + a_1f_1(X) \colon a_0,a_1\in \F_{q^n} \}\subseteq \cC$. If $\cN = L(\cC)$ then $\{f_1(aX) \colon a \in \F_{q^n} \}\subseteq \bar \cC$ which means
	$\{f_0(a_0X) + f_1(a_1X) \colon a_0,a_1\in \F_{q^n} \}\subseteq \cC$. If $|\cC|=q^{2n}$, we have proved \eqref{eq:normalized_C_1}; otherwise we continue this process by choosing a suitable $t_2\notin \{t_0,t_1\}$ and so on.  After finite steps, we obtain $|\cC|=q^{kn}$ and \eqref{eq:normalized_C_1}.
	
	Now we prove \eqref{eq:normalized_C_2}, so suppose that the other idealiser $\cM$ is also isomorphic to $\F_{q^n}$. As we already mentioned, we may assume
	\begin{equation}\label{eq:N'=all}
	\cM=\{aX \colon a\in \F_{q^n} \}.
	\end{equation}
	
	By \eqref{eq:normalized_C_1},
	\[f(X)=c_{0}X^{q^{t_0}} + \sum^{n-1}_{j\notin \{t_0,\ldots,t_{k-1}\}} g_j(c_0,0,\ldots,0) X^{q^j}\in \cC,\]
	for each $c_0\in\F_{q^n}$. For any $b\in \F_{q^n}^*$,  it is clear that $\varphi_2(X):= bX\in L(\cC)$ and $\varphi_1(X):=b^{-q^{t_0}}X\in R(\cC)$. Then
	\[ \varphi_1\circ f\circ \varphi_2(X) = c_0X^{q^{t_0}} + \sum^{n-1}_{j\notin \{t_0,\ldots,t_{k-1}\}} g_j(c_0,0,\ldots,0)b^{q^{j}-q^{t_0}} X^{q^j}\in \cC.\]
	Since $f$ is the unique element in $\cC$ associated with $(a_0,\ldots,a_{k-1})=(c_0, 0,\ldots, 0)$ we have
	\[g_j(c_0,0,\ldots,0)b^{q^{j}-q^{t_0}}= g_j(c_0,0,\ldots,0)\]
	for every $b\in \F_{q^n}$, which implies that $g_j(c_0,0,\ldots,0)=0$ for every $j\notin \{t_0,\ldots,t_{k-1}\}$ and for each $c_0\in \F_{q^n}$. Similarly, we can prove that  $g_j(0,\ldots,c_i,\ldots,0)=0$ for every $i\in \{0,\ldots,k-1\}$, $j\notin \{t_0,\ldots, t_{k-1}\}$ and $c_i\in \F_{q^n}$. Since $g_j(a_0,\ldots, a_{k-1})=g_j(a_0,0,\ldots) + g_j(0, a_1, 0,\ldots) + \dots + g_j(0,\ldots, a_{k-1})$, $g_j$ is the zero map for each $j$. Therefore we obtain \eqref{eq:normalized_C_2}.
\end{proof}

The next result shows how to handle the equivalence problem of MRD codes given as in \eqref{eq:normalized_C_2}.

\begin{theorem}\label{th:equivalence}
	Let $\Lambda_1$ and $\Lambda_2$ be two $k$-subsets of $\{0,\dots, n-1\}$. Define
	\[\cC_j=\left\{ \sum_{i\in \Lambda_j}a_{i} X^{q^{i}}\colon a_i\in \F_{q^n} \right\}\]
	for $j=1,2$. Then $\cC_1$ and $\cC_2$ are equivalent if and only if
	\begin{equation}\label{eq:lambda12}
		\Lambda_2=\Lambda_1+s:=\{i+s \pmod{n} \colon i\in \Lambda_1\}
	\end{equation}
	for some $s\in \{0,\cdots,n-1 \}$.
\end{theorem}
\begin{proof}
	The if part is trivial since $\Lambda_2=\Lambda_1+s$ implies $\cC_2=X^{q^s} \circ \cC_1$.
	Assume that $\cC_1$ and $\cC_2$ are equivalent. Let $\tau=(\varphi_1,\varphi_2,\rho)$ denote an  equivalence map from $\cC_1$ to $\cC_2$, i.e.\
	\[\{\varphi_1\circ f^\rho \circ \varphi_2\colon f\in \cC_1\}=\cC_2.\]
	For every $j\in \{0,\cdots, n-1\}$, let $\cD_j=\{aX^{q^j} \colon a\in \F_{q^n}\}$. Define
	\[I_j= \{i \colon \text{the coefficient of $X^{q^i}$ in } \varphi_1 \circ g^\rho \circ \varphi_2(X) \text{ is non-zero for some }g\in \cD_j\}.\]
	Since $\varphi_1 \circ g^\rho \circ \varphi_2$ is the zero polynomial only when $g$ is the zero polynomial, it follows that $I_j\neq \emptyset$ for each $j$.
	By \cite[Lemma 4.5]{lunardon_generalized_2015}, for any $j,l\in\{0,\cdots, n-1\}$,
	\[I_{l} = I_{j}+l-j:=\{i+l-j \pmod{n}\colon i\in I_j\}.\]
	If $l\in \Lambda_1$, then $\cD_l \subseteq \cC_1$ and hence $I_l\subseteq \Lambda_2$.
	Take any $s\in I_0$ and $l\in \Lambda_1$, then $s+l \in I_0+l=I_l \subseteq \Lambda_2$ and hence by $|\Lambda_1|=|\Lambda_2|=k$ we obtain $\Lambda_2=\Lambda_1+s$.	
\end{proof}

\subsection{Links with Moore Matrices}
It is clear that generalized Gabidulin codes and codes equivalent to them have maximum idealisers. It is not difficult to verify that they are actually the only known examples with this property. Hence, it is natural to ask whether there are MRD codes, inequivalent to the generalized Gabidulin codes, which have maximum idealisers. If they exist, can we classify them?

This question also has an interesting link with Moore matrices and Moore determinants which were introduced by Moore \cite{moore_two-fold_1896} in 1896.

Let $q$ be a prime power and take two positive integers, $n$ and $s$, with $\gcd(n,s)=1$. Put $\sigma:=q^s$. For
$A:=\{\alpha_0,\alpha_1,\ldots,\alpha_{k-1}\} \subseteq \F_{q^{n}}$, $k\leq n$,  a \emph{square Moore matrix} is defined as
\begin{equation}
\label{Moore}
M_{A,\,\sigma}:=\left(
\begin{matrix}
\alpha_0 & \alpha_0^{\sigma} & \cdots & \alpha_0^{\sigma^{k-1}} \\
\alpha_1 & \alpha_1^{\sigma} & \cdots & \alpha_1^{\sigma^{k-1}} \\
\vdots&  \vdots & \ddots &  \vdots \\
\alpha_{k-1} & \alpha_{k-1}^{\sigma} & \cdots & \alpha_{k-1}^{\sigma^{k-1}}
\end{matrix}
\right),
\end{equation}
which is a $\sigma$-analogue for the Vandermonde matrix. When it is clear from the context, then $\sigma$ will be omitted and we will simply write $M_A$. When $s=1$, then the determinant of $M$ can be expressed as
\begin{equation}
\label{detMoore}
\det(M_A)=\prod_{\mathbf{c}} (c_0\alpha_0 + c_1\alpha_1 +\cdots c_{k-1}\alpha_{k-1}),
\end{equation}
where $\mathbf{c}=(c_0,c_1,\cdots,c_{k-1})$ runs over all direction vectors in $\F_q^k$, or equivalently we can say that $\mathbf{c}$ runs over $\PG(k-1,q)$. We call $\det(M_A)$ the \emph{Moore determinant}.
It is not difficult to see that the following generalization also holds.
(In Remark \ref{prof} we will show how this result follows also from our Theorem \ref{connection}.)

\begin{theorem}\label{Mooremtx}
	Assume that $s$ satisfies $\gcd(s,n)=1$. For any $A=\{\alpha_0,\alpha_1,\ldots,\alpha_{k-1}\} \subseteq \F_{q^n}$, $k\leq n$, the elements of $A$ are $\F_q$-linearly dependent if and only if $\det(M_{A})=0$.
\end{theorem}

Assume $\gcd(s,n)=1$ and take any set of pairwise distinct integers $\cT=\{t_0,t_1,\ldots,t_{k-1}\}$ with $0\leq t_0<t_1<\ldots<t_{k-1}<n$ and $A=\{\alpha_0,\alpha_1,\ldots,\alpha_{k-1}\} \subseteq \F_{q^n}$, $k\leq n$.
Put $\sigma=q^s$ and let
\begin{equation}
\label{Moore2}
M_{\cT,\,A,\,\sigma}:=\left(
\begin{matrix}
\alpha_0^{\sigma^{t_0}} & \alpha_0^{\sigma^{t_1}} & \cdots & \alpha_0^{\sigma^{t_{k-1}}} \\
\alpha_1^{\sigma^{t_0}} & \alpha_1^{\sigma^{t_1}} & \cdots & \alpha_1^{\sigma^{t_{k-1}}} \\
\vdots&  \vdots & \ddots &  \vdots \\
\alpha_{k-1}^{\sigma^{t_0}} & \alpha_{k-1}^{\sigma^{t_1}} & \cdots & \alpha_{k-1}^{\sigma^{t_{k-1}}}
\end{matrix}
\right).
\end{equation}

As before, $\sigma$ will be omitted when it is clear from the context.
It is easy to see that if the elements of $A$ are $\F_q$-linearly dependent, then $\det(M_{\cT, A})=0$. Regarding the other direction we have the following.

\begin{theorem}\label{connection}
	Assume that $s$ satisfies $\gcd(s,n)=1$ and put $\sigma=q^s$. The set of $q$-polynomials
	\begin{equation}
	\label{codice}
	\{a_0 X^{\sigma^{t_0}}+a_1 X^{\sigma^{t_1}}+\ldots+a_{k-1}X^{\sigma^{t_{k-1}}}\colon a_0,a_1,\ldots,a_{k-1}\in \F_{q^n}\}
	\end{equation}
	is an MRD code (with maximum idealisers) if and only if
	for any $A=\{\alpha_0,\alpha_1,\ldots,\alpha_{k-1}\} \subseteq \F_{q^n}$, $k\leq n$, $\det(M_{\cT,A})=0$ implies that the elements of $A$ are $\F_q$-linearly dependent.
\end{theorem}
\begin{proof}
	Note that $\det(M_{\cT,A})=0$ for some $k$-subset $A\subseteq \F_{q^n}$ if and only if the columns of $M_{\cT,A}$ are dependent over $\F_{q^n}$ which holds if and only if there exist $a_0,a_1,\ldots,a_{k-1}\in \F_{q^n}$, not all of them zero, such that
	\[\sum_{j=0}^{k-1} \alpha_i^{\sigma^{t_j}} a_j =0\]
	holds for $i\in \{0,1,\ldots,k-1\}$. Equivalently, the elements of $A$ are roots of
	\begin{equation}
	\label{need}
	a_0X^{\sigma^{t_0}}+a_1X^{\sigma^{t_1}}+\ldots+a_{k-1}X^{\sigma^{t_{k-1}}}.
	\end{equation}
	
	If \eqref{codice} is an MRD code, then \eqref{need} cannot have $q^k$ roots and hence for any
	$k$-subset $A$ of $\F_q$-linearly independent elements we obtain $\det(M_{\cT,A})\neq 0$.
	
	On the other hand, if $\cT$ has been choosen such that $\det(M_{\cT,A})=0$ implies the $\F_q$-dependence of the elements in $A$ for any $k$-subset $A\subseteq \F_{q^n}$, then the non-zero polynomials of \eqref{codice} have less than $q^k$ roots and hence \eqref{codice} is an MRD code.
	
	By Theorem \ref{th:classification}, if $\eqref{codice}$ is an MRD code, then it has maximum idealisers.
\end{proof}


\begin{remark}
	\label{prof}
	It follows from Theorem \ref{connection} with $t_i=i$ for $i\in \{0,1,\ldots,k-1\}$ that Moore's Theorem \ref{Mooremtx} is equivalent to the fact that generalized Gabidulin codes are MRD codes.
\end{remark}

For a $k$-subset $\cT$ of $\{0,1,\ldots,n-1\}$, let $V_\cT$ denote the hypersurface of $\PG(k-1,\K)$, where $\K$ is the algebraic closure of $\F_{q}$, defined by the polynomial
\[
\det\left(\begin{matrix}
	X_0^{\sigma^{t_0}} & X_0^{\sigma^{t_1}} & \cdots & X_0^{\sigma^{t_{k-1}}} \\
	X_1^{\sigma^{t_0}} & X_1^{\sigma^{t_1}} & \cdots & X_1^{\sigma^{t_{k-1}}} \\
	\vdots&  \vdots & \ddots &  \vdots \\
	X_{k-1}^{\sigma^{t_0}} & X_{k-1}^{\sigma^{t_1}} & \cdots & X_{k-1}^{\sigma^{t_{k-1}}} \\
\end{matrix} \right)\in \F_{q}[X_0,X_1,\ldots,X_{k-1}].
\]

The following will be used in Section \ref{sec:nonexistence} to prove the nonexistence result.

\begin{theorem}
	\label{confront}
	Fix $\sigma=q^s$ where $s$ is an integer such that $\gcd(s,n)=1$.
	Let $\cS=\{s_0,s_1,\ldots,s_{k-1}\}$ and $\cT=\{t_0,t_1,\ldots,t_{k-1}\}$ be two subsets of $\{0,1,\ldots,n-1\}$ and suppose that
	\[\cC_{\cT}:=\{a_0X^{\sigma^{t_0}}+a_1X^{\sigma^{t_1}}+\ldots+a_{k-1}X^{\sigma^{t_{k-1}}}\colon a_0,a_1,\ldots,a_{k-1}\in \F_{q^n}\}\]
	is an MRD code.
	Then
	\[\cC_{\cS}:=\{a_0X^{\sigma^{s_0}}+a_1X^{\sigma^{s_1}}+\ldots+a_{k-1}X^{\sigma^{s_{k-1}}}\colon a_0,a_1,\ldots,a_{k-1}\in \F_{q^n}\}\]
	is an MRD code if and only if there are no $\F_{q^n}$-rational points in $V_{\cS}\setminus V_{\cT}$.
\end{theorem}
\begin{proof}
	According to Theorem \ref{connection} the $\F_{q^n}$-rational points of $V_{\cT}$ are
	\[\cL:=\{\langle (\alpha_0,\alpha_1,\ldots,\alpha_{k-1})\rangle_{\F_{q^n}}\in \PG(k-1,q^n) \colon \dim\langle \alpha_0,\alpha_1,\ldots,\alpha_{k-1}\rangle_{\F_q}<k\}.\]
	If $\cC_{\cS}$ is also an MRD code, then again from Theorem \ref{connection} the set of $\F_{q^n}$-rational points of $V_{\cS}$ coincides with the point set $\cL$. On the other hand if
	there exists $\langle(\alpha_0,\alpha_1,\ldots,\alpha_{k-1})\rangle_{\F_{q^n}}\in V_{\cS}\setminus V_{\cT}$, then $\dim\langle \alpha_0,\alpha_1,\ldots,\alpha_{k-1}\rangle_{\F_q}=k$ and with $A=\{\alpha_0,\alpha_1,\ldots,\alpha_{k-1}\}$ we have $\det(M_{\cS,A})=0$. Theorem \ref{connection}   yields that $\cC_{\cS}$ is not an MRD code.
\end{proof}

\section{Constructions and classifications}\label{sec:constructions}

In this section our aim is to classify $\F_q$-linear MRD codes with maximum idealisers in $\F_q^{n \times n}$ with $n\leq 9$. In terms of linearized polynomials, by Theorem \ref{th:classification} it is equivalent to find $k$-subsets $\cT:=\{t_0,t_1,\ldots, t_{k-1}\}$ of $\{0,1,\ldots,n-1\}$ such that the non-zero polynomials in
\[\cC_{\cT}:=\{a_0X^{q^{t_0}}+a_1X^{q^{t_1}}+\ldots+a_{k-1}X^{q^{t_{k-1}}}\colon a_0,a_1,\ldots,a_{k-1}\in \F_{q^n}\}\]
have at most $q^k$ roots.

Clearly, if $k=1$, then we obtain generalized Gabidulin codes with minimum distance $n$.

\begin{proposition}
	\label{k=2}
Let $\cT=\{t_0,t_1,\ldots, t_{k-1}\}\subseteq\{0,1,\ldots,n-1\}$. If $\cC_{\cT}$ is an MRD code then
$\gcd(t_i-t_j,n)<k$ for each $i\neq j$, $i,j\in \{0,1,\ldots,k-1\}$.
\end{proposition}
\begin{proof}
We may assume $t_j<t_i$ and put $s=\gcd(t_j-t_i,n)$. It is enough to observe that the elements of $\F_{q^s}\subseteq \F_{q^n}$ are roots of $(X^{q^{t_i-t_j}}-X)^{q^{t_j}}\in \cC_{\cT}$ and hence if $s\geq k$, then $\cC_{\cT}$ is not an MRD code.
\end{proof}

If $k=2$, then by Proposition \ref{k=2} we have to consider polynomials of the form
\[\{a_0X^{q^{t_0}}+a_1X^{q^{t_1}} \colon a_0,a_1 \in \F_{q^n}\},\]
with $\gcd(t_1-t_0,n)=1$. These codes are clearly equivalent to generalized Gabidulin codes.

Applying Delsarte dual operation we may always assume $k\leq n/2$, since $\cC_{\cT}^\perp=\cC_{\cT'}$ where $\cT'=\{0,1,\ldots,n-1\}\setminus \cT$.
As $\cC_{\cT}$ is equivalent to $\cC_{\cT'}$ (cf. Theorem \ref{th:equivalence}) for every $\cT'=\cT+s:=\{t+s \pmod n \colon t\in \cT\}$, we may also assume $0\in \cT$.

Applying now the adjoint operation we may further assume that for $k>1$ there exists $1\leq i \leq n/2$ such that $i\in \cT$. This is because if $0\in \cT$ then $\cC_{\cT}^\top=\cC_{\cT'}$ where $\cT'=\{0\} \cup \{ n-i \colon i\in \cT,\, i\neq 0\}$.

It follows that for $n\leq 5$ the MRD codes with both idealisers isomorphic to $\F_{q^n}$ are equivalent to generalized Gabidulin codes.

Now consider $n=6$ and $k=3$. It is enough to consider polynomial subspaces of the form
\[\{a_0X+a_1X^{q^{t_1}}+a_2X^{q^{t_2}} \colon a_0,a_1,a_2 \in \F_{q^6}\},\]
with $t_1\in \{1,2\}$ and $t_1<t_2$. From Proposition \ref{k=2} we have $\gcd(t_2,6)\leq 2$ and $\gcd(t_2-t_1,6)\leq 2$. If $t_1=1$ then we get $t_2\in\{2,5\}$ and both cases yield codes equivalent to  Gabidulin codes. If $t_1=2$ then $t_2=4$ but then $\Tr_{q^6/q^2}(X)$ is in the code, a contradiction since it has $q^4$ roots in $\F_{q^6}$. Thus we have proved the following.

\begin{proposition}
	\label{n=6}
If $n\leq 6$ then MRD codes with both idealisers isomorphic to $\F_{q^n}$ are equivalent to generalized Gabidulin codes.
\end{proposition}

Using a similar argument together with Theorem \ref{th:equivalence}, we can exclude most of the possibilities also for $n=7,8,9$ and obtain that, up to $\perp$ and $\top$ operations if an MRD code $\cC_{\cT}$ with $\cT\subseteq \{0,1,\ldots,n-1\}$ has maximum left and right idealisers and it is not equivalent to generalized Gabidulin codes then up to equivalence it has to have one of the following form:
\begin{enumerate}
	\item $n\in\{7,8\}$, $k=3$ and $\cT=\{0,1,3\}$,
	\item $n=9$, $k=4$ and $\cT=\{0,s,2s,4s\}$, where $s\in \{1,4,7\}$ and the elements of $\cT$ are considered modulo 9.
\end{enumerate}

As we will see, in the first case we have MRD codes under certain conditions on $q$ while in the second case we never obtain MRD codes.

We recall the following result on $q$-polynomials which we will use frequently.
Let $f(X)=\sum_{i=0}^{n-1}a_i X^{q^i}$ with $a_0,a_1,\ldots,a_{n-1}\in\F_{q^n}$ and let $D_f$ denote the associated \emph{Dickson matrix} (or \emph{$q$-circulant matrix})
\[D_f:=
\begin{pmatrix}
a_0 & a_1 & \ldots & a_{n-1} \\
a_{n-1}^q & a_0^q & \ldots & a_{n-2}^q \\
\vdots & \vdots & \vdots & \vdots \\
a_1^{q^{n-1}} & a_2^{q^{n-1}} & \ldots & a_0^{q^{n-1}}
\end{pmatrix}
.\]
Then the rank of $D_f$ equals the rank of $f$ viewed as an $\F_q$-linear transformation of $\F_{q^n}$, see for example \cite{wuliu}.

\subsection{The $n=7$ case}

\begin{theorem}	\label{thm1}
    The set of $q$-polynomials
	\begin{equation}
	\label{Cn7}
\cC_7:=\{a_0X+a_1X^q+a_2X^{q^3} \colon a_0,a_1,a_2\in \F_{q^7}\}
	\end{equation}
	is an $\F_q$-linear MRD code with left and right idealisers isomorphic to $\F_{q^7}$ if and only if $q$ is odd.
	Moreover, $\cC_7$ is not equivalent to the previously known MRD codes.
\end{theorem}
\begin{proof}
	The Dickson matrix associated with $f(X)=X+X^q+X^{q^3}\in \F_{q^7}[X]$ is
	\[\begin{pmatrix}
	1 & 1 & 0 & 1 & 0 & 0 & 0 \\
	0 & 1 & 1 & 0 & 1 & 0 & 0 \\
	0 & 0 & 1 & 1 & 0 & 1 & 0 \\
	0 & 0 & 0 & 1 & 1 & 0 & 1 \\
	1 & 0 & 0 & 0 & 1 & 1 & 0 \\
	0 & 1 & 0 & 0 & 0 & 1 & 1 \\
	1 & 0 & 1 & 0 & 0 & 0 & 1 \\
	\end{pmatrix}.
	\]
	This matrix can also be viewed as the incidence matrix of the points and lines of $\PG(2,2)$. It is well-known, and also easy to see, that it has rank four over $\F_2$, hence $f(X)$ has $q^3$ roots, i.e. $\cC_7$ is not an MRD code.
	
	Now let $q$ be odd and suppose to the contrary that $\cC_7$ is not an MRD code. Then there exist $\alpha_1, \alpha_2, \alpha_3 \in \F_{q^7}$ such that $\alpha_1 X + \alpha_2 X^q + \alpha_3 X^{q^3}$ has $q^3$ roots. Clearly these roots form an $\F_q$-subspace of $\F_{q^7}$, let $u_1,u_2,u_3$ be an $\F_q$-basis for this subspace.
	
	Let $\sigma$ denote the collineation of $\PG(2,q^7)$ defined by the following semilinear map of $\F_{q^7}^3$: $(x_1,x_2,x_3) \mapsto (x_1^q,x_2^q,x_3^q)$.
	Let $\Sigma \cong \PG(2,q)$ denote the points of $\PG(2,q^7)$ fixed by $\sigma$.
	Define $P:=\la(u_1,u_2,u_3)\ra_{\F_{q^7}}$ and
note that $P\notin \Sigma$, otherwise $\lambda (u_1,u_2,u_3)=(u_1^q,u_2^q,u_3^q)$ for some $\lambda\in \F_{q^7}^*$, a contradiction since this would mean that $u_1^{q-1}=u_2^{q-1}=u_3^{q-1}$, i.e. $\dim \la u_1,u_2,u_3\ra_{\F_q} = 1$.
It follows that $P$ lies on an orbit of length seven of $\sigma$.
	
	The scalars $\alpha_1, \alpha_2, \alpha_3$ show that the columns of the matrix
	\[M:=
	\begin{pmatrix}
	u_1 & u_1^q & u_1^{q^3} \\
	u_2 & u_2^q & u_2^{q^3} \\
	u_3 & u_3^q & u_3^{q^3}
	\end{pmatrix}	\]
	are $\F_{q^7}$-linearly dependent and hence also the rows of $M$ are $\F_{q^7}$-linearly dependent, which shows that there exists a line $\ell$ of $\PG(2,q^7)$ which is incident with $P$, $P^\sigma$ and $P^{\sigma^3}$.
	
	First we show that $\ell$ is not a line of $\Sigma$, which is equivalent to say $\ell \neq \ell^\sigma$.
	Suppose the contrary, then $\ell$ has an equation $a_1X_1+a_2X_2+a_3X_3=0$ where $X_1,X_2,X_3$ denote the homogeneous coordinates for points of $\PG(2,q^7)$ and $a_1,a_2,a_3 \in \F_q$. A contradiction since $\dim \la u_1,u_2,u_3\ra_{\F_q}=3$.
	
	Next we show that $\ell$ cannot be tangent to $\Sigma$. Suppose to the contrary that $\ell \cap \Sigma= \{Q\}$ for some point $Q$.
	Then $Q\in \ell\cap \ell^\sigma = \{P^\sigma\}$, a contradiction since $\{P, P^\sigma, P^{\sigma^2},\ldots, P^{\sigma^6}\}$ are not fixed by $\sigma$ hence $P^\sigma=Q$ cannot be a point of $\Sigma$.
	
	Thus $\ell$ lies on an orbit of length 7 of $\sigma$ and since $\{0,1,3\}$ is a cyclic $(7,3,1)$-difference set of $\mathbb{Z}_7$, the cyclic group of order 7 (written additively), we have that the points $\{P, P^\sigma, P^{\sigma^2},\ldots, P^{\sigma^6}\}$ and lines $\{\ell, \ell^\sigma,\ldots,\ell^{\sigma^6}\}$ form a Fano subplane inside $\PG(2,q^7)$. However, it is well known that a Fano plane cannot be embedded in $\PG(2,q)$ if $q$ is odd. Thus we get a contradiction.
	
The last part follows from Theorem \ref{th:equivalence} and from the fact that the only known MRD codes with maximum left and right idealisers are equivalent to the generalized Gabidulin codes.
\end{proof}

As observed in Section \ref{sec:pre}, the Delsarte dual operation preserves the equivalence relations between MRD codes. Hence we have the following result.
\begin{corollary}\label{cor1}
    The set of $q$-polynomials
	\begin{equation}
	\label{Cn7dd}
\cC_7':=\{a_0X+a_1X^{q^3}+a_2X^{q^5}+a_3X^{q^6} \colon a_0,a_1,a_2,a_3\in \F_{q^7}\}
	\end{equation}
	is an $\F_q$-linear MRD code with left and right idealisers isomorphic to $\F_{q^7}$ if and only if $q$ is odd.
	Moreover, $\cC_7'$ is not equivalent to the previously known MRD codes.
\end{corollary}

\subsection{The $n=8$ case}

\begin{theorem}
	\label{thm2}
	The set of $q$-polynomials
	\begin{equation}
	\label{Cn8}
	\cC_8:=\{a_0X+a_1X^q+a_2X^{q^3} \colon a_0,a_1,a_2\in \F_{q^8}\}
	\end{equation}
	is an $\F_q$-linear MRD code with left and right idealisers isomorphic to $\F_{q^8}$ if and only if $q\equiv 1 \pmod 3$.
	Moreover, $\cC_8$ is not equivalent to the previously known MRD codes.
\end{theorem}
\begin{proof}
	First suppose $q \not\equiv 1 \pmod 3$ and choose $a$ such that $1+a+a^2=0$. If $q\equiv -1 \pmod 3$, then $a\in\F_{q^2}\setminus\F_q$ and $a^q=1/a$. If $q \equiv 0 \pmod 3$, then $a=1$. Note that the Dickson matrix associated with $X+X^q+aX^{q^3}\in \F_{q^8}[X]$ is
	
	\[M:=\begin{pmatrix}
	1 & 1 & 0 & a & 0 & 0 & 0 & 0 \\
	0 & 1 & 1 & 0 & 1/a & 0 & 0 & 0 \\
	0 & 0 & 1 & 1 & 0 & a & 0 & 0 \\
	0 & 0 & 0 & 1 & 1 & 0 & 1/a & 0 \\
	0 & 0 & 0 & 0 & 1 & 1 & 0 & a \\
	1/a & 0 & 0 & 0 & 0 & 1 & 1 & 0 \\
	0 & a & 0 & 0 & 0 & 0 & 1 & 1 \\
	1 & 0 & 1/a & 0 & 0 & 0 & 0 & 1 \\
	\end{pmatrix}
	\]
	whose last five columns are linearly independent and hence the rank of $M$ is at least $5$.
	
	If the characteristic of $\F_q$ is $3$, then the rows of $M$ are orthogonal to the rows of
	\[
	\begin{pmatrix}
	2 & 0 & 1 & 1 & 2 & 1 & 0 & 0 \\
	2 & 2 & 1 & 2 & 0 & 0 & 1 & 0 \\
	0 & 2 & 2 & 1 & 2 & 0 & 0 & 1 \\
	\end{pmatrix},
	\]
	which is a matrix of rank $3$. It follows that in this case the rank of $M$ is $5$.
	
	On the other hand, if $q \equiv -1 \pmod 3$, then the matrix
	\[\begin{pmatrix}
	1 & a & a^2 & a & a & a^2 & -2 a^2 & a^2 \\
	a & 1 & a^2 & a & a^2 & a^2 & a & -2 a \\
	-2 a^2 & a^2 & 1 & a & a^2 & a & a & a^2 \\
	\end{pmatrix}
	\]
	has rank three and its rows are orthogonal to the rows of $M$, thus $M$ has rank $5$.
\medskip
	
	Now let $q \equiv 1 \pmod 3$ and suppose to the contrary that $\cC_8$ is not an MRD code.
	Then arguing as in the proof of Theorem \ref{thm1},  there exist $\F_q$-linearly independent elements  $u_1,u_2,u_3\in\F_{q^8}$ and a line $\ell$ of $\PG(2,q^8)$ incident with $P:=\la(u_1,u_2,u_3)\ra$ and with
	$P^\sigma, P^{\sigma^3}$, where $\sigma$ is the collineation of $\PG(2,q^8)$ defined by the semilinear map $(x_1,x_2,x_3) \mapsto (x_1^q,x_2^q,x_3^q)$. Also, let $\Sigma \cong \PG(2,q)$ denote the set of points of $\PG(2,q^8)$ fixed by $\sigma$. Since $P,P^\sigma,P^{\sigma^3}$ are three different points and since $\ell \cap \Sigma=\emptyset$, $P^{\sigma^2}$ and $P^{\sigma^5}$ are two further points, which are not incident with $\ell$. So, if $T:=\langle P^\sigma,P^{\sigma^2}\rangle\cap\langle P, P^{\sigma^5}\rangle$, then $P$, $P^{\sigma^2}$, $P^{\sigma^3}$ and $T$ are four points no three of which are collinear. Hence, there exists a projectivity $\varphi$ of $\PG(2,q^8)$ such that
	\[
	P^{\varphi}=\la (0,0,1) \ra=:P_0, \quad P^{\sigma^3\varphi}=\la (0,1,0)\ra=:P_3, \quad P^{\sigma^2\varphi}=\la (1,0,0) \ra=:P_2
	\]
	and $\la (1,1,1)\ra$ is the point $T^\varphi$. In this way
	\[
	P^{\sigma\varphi}=\la (0,1,1)\ra=:P_1, \quad P^{\sigma^5\varphi}=\la (1,1,0) \ra=:P_5, \quad
	P^{\sigma^6\varphi}=\la (a,a,1) \ra=:P_6
	\]
	 for some $a\in \F_{q^8}^*$. Also, elementary calculations show
	\[
	P^{\sigma^7\varphi}=\la (a,0,1-a) \ra=:P_7, \quad \mbox{ and } \quad P^{\sigma^4\varphi}=\la(1,1-a,1-a)\ra=:P_4.
	\]
	Since $P_3,P_4,P_6$ are collinear, it follows that
	\begin{equation}
	\label{root}
	a^2-a+1=0,
	\end{equation}
	and hence, since $q\equiv 1\pmod 3$, we get $a\in\F_q$.
	Let $\bar\sigma=\varphi\circ\sigma\circ\varphi^{-1}$. Then $\bar\sigma$ is a collineation of order 8 of $\PG(2,q^8)$ and it is induced by a
	semilinear map of this form
	\[(x_1,x_2,x_3) \mapsto \left(\sum_{j=1}^3 a_{1j}x_j^q,\sum_{j=1}^3 a_{2j}x_j^q,\sum_{j=1}^3 a_{3j}x_j^q\right),\]
	with $(a_{ij})$ a non-singular $3\times 3$ matrix over $\F_{q^8}$.
	By construction, it is easy to see that $P_i^{\bar\sigma}=P_{i+1}$, for $i=0,\dots,7\, \pmod 8$. Direct computations for $i=0,1,2,4$ show that
	up to a scalar of $\F_{q^8}^*$
	\[
	(a_{ij})=\left(
	\begin{array}{ccc}
	0  & 1  & 0  \\
	1-a  & 1-a  & a-1  \\
	0 &  1-a  & a-1
	\end{array}
	\right)
	\]
	and from $P_5^{\bar\sigma}=P_{6}$ we get $1=2-2a$.
	This clearly cannot hold if $q$ is even, while for $q$ odd it gives $a=1/2$ which does not satisfy \eqref{root}, a contradiction.
	
	The last part follows as in Theorem \ref{thm1}.
\end{proof}

Again, since the Delsarte dual operation preserves the equivalence relations between MRD codes, we have the following result.
\begin{corollary}\label{cor2}
    The set of $q$-polynomials
	\begin{equation}
	\label{Cn8dd}
\cC_8':=\{a_0X+a_1X^{q^2}+a_2X^{q^3}+a_3X^{q^4}+a_4X^{q^5} \colon a_0,a_1,a_2,a_3,a_4\in \F_{q^8}\}
	\end{equation}
	is an $\F_q$-linear MRD code with left and right idealisers isomorphic to $\F_{q^8}$ if and only if $q\equiv 1 \pmod 3$.
	Moreover, $\cC_8'$ is not equivalent to the previously known MRD codes.
\end{corollary}

\subsection{The $n=9$ case}

For $s\in \{1,4,7\}$ consider the rank codes
\[\cD_s:=\{a_0X+a_1X^{q^s}+a_2X^{q^{2s}}+a_3X^{q^{4s}} \colon a_0,a_1,a_2,a_3 \in \F_{q^9}^*\}.\]
First we show that $\cD_1$ is not an MRD code.

Put $f(X):=-X+(1+c^{-q})X^q+cX^{q^2}-X^{q^4}\in \cD_1$ with $c\in \F_{q^3}^*$ such that $\Tr_{q^3/q}(1/c)=-2$ and $\N_{q^3/q}(1/c)=-1$. Here $\N_{q^n/q}(x)=x^{1+q+\ldots+q^{n-1}}$ denotes the norm of $x\in \F_{q^n}$ over $\F_q$. By \cite[Theorem 5.3]{Moisio} we can find such an element $c$ in $\F_{q^3}^*$.
Let $D_f=(d_{ij})$ denote the Dickson matrix associated with $f$.
Substituting $-c^{-q-1}$ for $c^{q^2}$ at positions $d_{35}$, $d_{68}$ and $d_{92}$ we obtain
\[D_f=
\begin{pmatrix}
-1 & \alpha & c & 0 & -1 & 0 & 0 & 0 & 0 \\
0 & -1 & \beta & c^q & 0 & -1 & 0 & 0 & 0\\
0 & 0 & -1 & \gamma & -c^{-q-1} & 0 & -1 & 0 & 0\\
0 & 0 & 0 & -1 & \alpha & c & 0 & -1 & 0 \\
0 & 0 & 0 & 0 & -1 & \beta & c^q & 0 & -1\\
-1 & 0 & 0 & 0 & 0 & -1 & \gamma & -c^{-q-1} & 0\\
0 & -1 & 0 & 0 & 0 & 0 & -1 & \alpha & c \\
c^q & 0 & -1 & 0 & 0 & 0 & 0 & -1 & \beta\\
\gamma & -c^{-q-1} & 0 & -1 & 0 & 0 & 0 & 0 & -1\\
\end{pmatrix},\]
with $\alpha=1+c^{-q}$, $\beta=-1-c^{-1}-c^{-q}$ and $\gamma=1+1/c$, where $\beta$ is obtained after substituting $-1-c^{-1}-c^{-q}$ for $1+c^{-q^2}$.
The $5\times 5$ submatrix $M$ formed by the first five rows and the first five columns of $D_f$ is triangular with non-zero entries on its diagonal, hence it is non-singular. Then the rank of $D_f$ is five if and only if all the $6\times 6$ submatrices of $D_f$ which contain $M$ are singular (this is an exercise in linear algebra and we omit its proof).
We have 16 such submatrices and we consider their determinants as polynomials in $c$. By calculation, it turns out that each of them is divisible by
\begin{equation}
\label{divisor}
c^{2q+2}-2c^{q+1}-c^q-c.
\end{equation}
Note that $\N_{q^3/q}(c)=-1$ and hence $\Tr_{q^3/q}(c^{q+1})=\Tr_{q^3/q}(1/c)\N_{q^3/q}(c)=2$.
Multiplying \eqref{divisor} by $c^{q^2}$ gives $-\Tr_{q^3/q}(c^{q+1})+2=0$ thus $D_f$ has rank five.
It follows that $f(X)$ has $q^4$ roots and hence $\cD_1$ is not an MRD code.

Now let
\[K:=
\begin{pmatrix}
 1 & 0 & 0 & 0 & 0 & 0 & 0 & 0 & 0
\\
0 & 0 & 0 & 0 & 0 & 0 & 0 & 1 & 0
\\
0 & 0 & 0 & 0 & 0 & 1 & 0 & 0 & 0
\\
0 & 0 & 0 & 1 & 0 & 0 & 0 & 0 & 0
\\
0 & 1 & 0 & 0 & 0 & 0 & 0 & 0 & 0
\\
0 & 0 & 0 & 0 & 0 & 0 & 0 & 0 & 1
\\
0 & 0 & 0 & 0 & 0 & 0 & 1 & 0 & 0
\\
0 & 0 & 0 & 0 & 1 & 0 & 0 & 0 & 0
\\
0 & 0 & 1 & 0 & 0 & 0 & 0 & 0 & 0
\end{pmatrix}.
\]
Since $f$ has coefficients in $\F_{q^3}$, it is easy to see that $KD_fK^{-1}$ is the Dickson matrix associated with
$-X+(1+c^{-q})X^{q^4}+cX^{q^8}-X^{q^7}\in \cD_4$
and $K^2D_fK^{-2}$ is the Dickson matrix associated with
$-X+(1+c^{-q})X^{q^7}+cX^{q^5}-X^q\in \cD_7$.
It follows that these two polynomials have $q^4$ roots as well and hence $\cD_4$ and $\cD_7$ are not MRD codes, and we have proven the following result.
\begin{proposition}
	\label{n=9}
	If $n=9$ then MRD codes with both idealisers isomorphic to $\F_{q^9}$ are equivalent to generalized Gabidulin codes.
\end{proposition}

\subsection*{Proof of Theorem \ref{thm:main}} The result follows from Proposition \ref{n=6}, the discussions after Proposition \ref{n=6}, Theorem \ref{thm1}, Corollary \ref{cor1}, Theorem \ref{thm2}, Corollary \ref{cor2} and  Proposition \ref{n=9}.\qed

\section{Nonexistence result}\label{sec:nonexistence}

\subsection{Main result of this section}
Generalizing the notation from \eqref{Cn7} and \eqref{Cn8} let
\begin{equation}
\label{Cngen}
\cC_n:=\{a_0X+a_1X^q+a_2X^{q^3} \colon a_0,a_1,a_2 \in \F_{q^n}\}.
\end{equation}

As we have seen in Section \ref{sec:constructions} the MRD codes of $\F_{q}^{n \times n}$, $n\leq 9$, which are not equivalent to the generalized Gabidulin codes but have maximum left and right idealisers are, up to adjoint and Delsarte dual operations, 
equivalent either to $\cC_7$ (for $q$ odd) or to $\cC_8$ (for $q \equiv 1 \pmod 3$). It is natural to ask whether the family $\cC_n$ contains new MRD codes for larger values of $n$. In this direction, we will prove the following result.

\begin{theorem}
	\label{th:nonexistence}
	For $n\geq 9$ and any prime power $q$, $\cC_n$ is not an MRD code.
\end{theorem}

To prove Theorem \ref{th:nonexistence}, we will need the following lemma.
\begin{lemma}
	\label{le:intersection_number_m_m1_coprime}
	\cite[Proposition 2]{janwa_double-error-correcting_1995}
	Let $F$ be  a polynomial in $\F_q[X,Y]$ and suppose that $F$ is not absolutely irreducible, that is, $F=AB$ where the coefficients of $A$ and $B$ are in the algebraic closure of $\F_q$. Let $P=(u,v)$ be a point in the affine plane $\mathrm{AG}(2,q)$ and write
	\[
	F(X+u,Y+v)=F_m(X,Y)+F_{m+1}(X,Y)+\cdots,
	\]
	where $F_i$ is zero or homogeneous of degree $i$ and $F_m\ne 0$. Assume that $F_m$ is completely reducible as a power of a linear polynomial and $gcd(F_{m},F_{m+1})=1$. Then $I(P, \mathcal{A}\cap \mathcal{B})=0$, where $\mathcal{A}$ and $\mathcal{B}$ are the curves defined by $A$ and $B$ respectively.
\end{lemma}

\begin{proof}[Proof of Theorem \ref{th:nonexistence}]
First, for $n=9$, it is easy to see that $X^{q^3}-X\in \cC_9$ has exactly $q^3$ roots which implies that $\cC_9$ is not MRD. In the rest of the proof we will assume $n\geq 10$.

We will apply Theorem \ref{confront} with $\cS=\{0,1,3\}$ and $\cT=\{0,1,2\}$. It gives us that $\cC_n$ is an MRD code if and only if $\cH \setminus \cW$ does not have $\F_{q^n}$-rational points, where $\cH$ and $\cW$ are projective curves defined by
\[H(X_0,X_1,X_2):=-X_0^{q^3}X_1^qX_2+X_0^qX_1^{q^3}X_2+X_0^{q^3}X_1X_2^q-X_0X_1^{q^3}X_2^q-X_0^qX_1X_2^{q^3}+X_0X_1^qX_2^{q^3}\]
and
\[W(X_0,X_1,X_2):=-X_0^{q^2}X_1^qX_2+X_0^qX_1^{q^2}X_2+X_0^{q^2}X_1X_2^q-X_0X_1^{q^2}X_2^q-X_0^qX_1X_2^{q^2}+X_0X_1^qX_2^{q^2},\]
respectively.


It is clear that $H(0,X_1,X_2)=W(0,X_1,X_2)=0$. Hence, we only have to investigate the points $\la(1,x,y)\ra$ for $x$ and $y\in \F_{q^n}$. By calculation,
\[H(1,X,Y)=(X^{q}-X^{q^3})(Y^{q}-Y) + (Y^{q^3}-Y^q)(X^{q}-X),\]
\[W(1, X, Y) =(X^{q}-X^{q^2})(Y^{q}-Y) + (Y^{q^2}-Y^q)(X^{q}-X).\]

Then to prove our assertion it is enough to show that the affine curve $\cV$ defined by
\begin{equation}\label{eq:curve_F}
V(X,Y):=\frac{H(1,X,Y)}{W(1,X,Y)} = \frac{(Y^q-Y)^{q^2-1}-(X^q-X)^{q^2-1}}{(Y^q-Y)^{q-1}-(X^q-X)^{q-1}} +1
\end{equation}
admits at least one $\F_{q^n}$-rational point $(x,y)$ which does not lie on the affine part of the curve $\cW$ defined by $W(1,X,Y)$.

	By calculation,
	\begin{equation}\label{eq:V_ex}
	V(X,Y) = \prod_{\gamma\in\F_{q^2}\setminus \F_{q}}((X^q-X) - \gamma(Y^q-Y))+1.
	\end{equation}
	
	It is not difficult to get an upper bound for the number of affine points in $\cV \cap \cW$. If a point $(x,y)$ is on $\cW$, then it satisfies one of the following conditions:
	\begin{enumerate}[label=(\alph*)]
		\item $x^q-x=0$, i.e.\ $x\in \F_q$;
		\item $y^q-y=0$, i.e.\ $y\in \F_q$;
		\item $x^q-x=\xi (y^q-y)$, where $\xi\in \F_q^*$.
	\end{enumerate}
	
	When $x\in \F_q$, $V(x,y)=  (y^q-y)^{q^2-q} +1$. It follows that $(y^q-y)^{q^2} = -(y^q-y)^q$ and $y\notin \F_q$. Hence $y\in \F_{q^2}\setminus \F_q$ and there are exactly $q(q^2-q)=q^3-q^2$ points $(x,y)$ of type (a) on $\cV \cap \cW$.
	
	When $y\in \F_q$, by symmetry, we get another $q^3-q^2$ points in $\cV\cap \cW$.
	
	When $x^q-x=\xi (y^q-y)$ with $\xi\in \F_q^*$,
	\begin{align*}
	V(x,y)&=\prod_{\gamma\in\F_{q^2}\setminus \F_{q}}(\xi-\gamma)(y^q-y)^{q^2-q} +1\\
	 &=\frac{\prod_{\gamma\in\F_{q^2}\setminus\{\xi\}}(\xi-\gamma)}{\prod_{\gamma\in\F_{q}\setminus\{\xi\}}(\xi-\gamma)}(y^q-y)^{q^2-q} +1\\
	&=(y^q-y)^{q^2-q} +1.
	\end{align*}
	This means that $y\notin \F_q$ and $y$ also satisfies
	\[(y^q-y)^q = y-y^q.\]
	Thus for given $\xi$, there are exactly $q^2-q$ solutions of $y$ and for each $y$, there are exactly $q$ solutions of $x$ for $x^q-x=\xi (y^q-y)$. As $\xi$ can be taken any value in $\F_q^*$, there are in total $q^2(q-1)^2$ points $(x,y)$ of type (c).
	
	Therefore we have proved that there are $B_a:=q^2(q-1)^2+2(q^3-q^2)=q^4-q^2$ affine points in $\cV \cap \cW$, hence the number of $\F_{q^n}$-rational affine points in $\cV \cap \cW$ is at most $q^4-q^2$.
	
	Let
	\begin{equation*}
	V^*(X,Y,T) = 	\prod_{\gamma\in\F_{q^2}\setminus \F_{q}}((X^q-XT^{q-1}) - \gamma(Y^q-YT^{q-1}))+T^{q^3-q^2}
	\end{equation*}
	be the homogenized polynomial of $V$. 	By considering the zeros of
	\begin{equation}\label{eq:V^*}
	V^*(X,1,0)=  \prod_{\gamma\in\F_{q^2}\setminus \F_{q}} (X^q-\gamma) = \left( \prod_{\gamma\in\F_{q^2}\setminus \F_{q}}(X-\gamma) \right)^q,	
	\end{equation}
	we see that the points at infinity of $\cV$ are $R_\gamma = (\gamma,1, 0)$ for $\gamma\in \F_{q^2}\setminus \F_q$. Hence there are $q^2-q$ points at infinity.
	
	Very recently,  Giulietti, Korchm\'{a}ros and Timpanella \cite{Giulietti_DGZ-curves_arxiv} also investigated this curve and they called it the Dickson-Guralnick-Zieve curve after the work \cite{Guralnick_automorphisms_2004} by Guralnick and Zieve, see also \cite{Borges}. They can show that this curve is absolutely irreducible \cite[Proposition 4.7]{Giulietti_DGZ-curves_arxiv} and the genus of $\cV$ is $g_q=\frac{1}{2}q(q-1)(q^3-2q-2)+1$ \cite[Theorem 4.10]{Giulietti_DGZ-curves_arxiv}. Moreover, by Lemmas 4.5 and 4.6 in \cite{Giulietti_DGZ-curves_arxiv}, each singular point of $\cV$ has a unique branch centered on it, which means the number $R_{q^n}$ of the $\F_{q^n}$-places of the associated function field of $\cV$ equals the number of $\F_{q^n}$-rational points of $\cV$ (for further details see \cite[Chapter 4]{HKT}). By the Hasse-Weil Theorem, one gets
	\[ \#\cV(\F_{q^{n}})=R_{q^n} \geq q^n+1 - 2g_q\sqrt{q^n}.\]

	Together with the total number $B_a$ of the affine points in $\cV\cap \cW$ and the $q^2-q$ points of $\cV$ at infinity, the existence of an affine $\F_{q^n}$-point $(x,y)$ on $\cV \setminus \cW$ is ensured whenever
	
	\begin{equation}\label{eq:enough_points}
	q^n+1 - 2g_q\sqrt{q^n} > q^4-q^2 + q^2-q=q^4-q.
	\end{equation}

	By plugging the value of $g_q$ into it, it is straightforward to check that \eqref{eq:enough_points} holds for $n\geq 10$.
	
In \cite{Giulietti_DGZ-curves_arxiv}, the authors proved the absolutely irreducibility by analyzing the branches of the curve. In the rest of our proof, we present an alternative proof only using B\'ezout's theorem, see for example \cite[Chapter 3]{HKT}. We assume that $\cV$ splits into two components $\cA$ and $\cB$ sharing no common irreducible component. Then we determine all possible singular points of $\cV$ and show that the sum of all intersection numbers of $\cA$ and $\cB$ equals $0$. Then by B\'ezout's theorem, we see that one of $\cA$ and $\cB$ must be a constant.
	
	It appears quite complicated to compute the affine singular points $(\alpha, \beta)$ of $\cV$ and the expansion of $V(X+\alpha, Y+\beta)$ directly. Instead, we investigate those for
\[U(X,Y)=-H(1,X,Y)=(X^{q^3}-X^q)(Y^q-Y) - (X^q-X)(Y^{q^3}-Y^q).\]
	By \eqref{eq:curve_F}, it is clear that
	\[V(X,Y)=\frac{U(X,Y)}{S(X,Y)},\]
	where $S(X,Y)=-W(1,X,Y)$. Hence every singular point of $\cV$ is also a singular point of the curve $\cU$ defined by $U$.
	
	By calculation,
	\[ \frac{\partial U(X,Y)}{\partial X} = Y^{q^3}-Y^q, \qquad   \frac{\partial U(X,Y)}{\partial Y} = -(X^{q^3}-X^q).\]
	It follows that every affine singular point $(x,y)$ of $\cU$ belongs to $\F_{q^2}\times \F_{q^2}$.
	
	When $\alpha,\beta\in \F_q$, by \eqref{eq:V_ex}, $(\alpha,\beta)$ is not a point on $\cV$. We only have to consider the points $(\alpha, \beta)\in \F_{q^2}^2\setminus \F_q^2$. By calculation,
	\begin{align*}
	&U(X+\alpha, Y+\beta) \\
	=& (\beta^q-\beta)(X^{q^3}-X^q) - (\alpha^q-\alpha) (Y^{q^3}-Y^q) + U(X,Y)\\
	=&(\beta-\beta^q)X^q - (\alpha-\alpha^q)Y^q+XY(X^{q-1}-Y^{q-1}) + \cdots\\
	=&(\bar{\beta}X -\bar{\alpha} Y )^q+XY(X^{q-1}-Y^{q-1}) - XY(X^{q^2-1} - Y^{q^2-1}) \\
	&+ (\bar{\alpha} Y-\bar{\beta} X)^{q^3} + (XY(X^{q^2-1}-Y^{q^2-1}))^q,
	\end{align*}
	where $\bar{\beta}^q = \beta-\beta^q$ and $\bar{\alpha}^q = \alpha-\alpha^q$. As $(\alpha,\beta)\notin \F_q\times \F_q$, $\bar{\beta}X -\bar{\alpha} Y\neq 0$.
	
	It is routine to compute that
	\begin{align*}
	S(X+\alpha, Y+\beta) =& (\bar{\alpha}Y - \bar{\beta}X) + (\bar{\beta}X - \bar{\alpha}Y)^{q^2} + XY(X^{q-1}-Y^{q-1})\\
	&+ (XY(X^{q-1}-Y^{q-1}))^q + XY(X^{q^2-1}-Y^{q^2-1}).
	\end{align*}
	
	As
	\[\bar{\beta}^q=-\bar{\beta}, ~\bar{\alpha}^q=-\bar{\alpha},  \]
	$\bar{\beta}X -\bar{\alpha} Y$ divides $XY(X^{q-1}-Y^{q-1})$ and $XY(X^{q^2-1}-Y^{q^2-1})$ for all $(\alpha,\beta)\in \F_{q^2}^2\setminus \F_{q}^2$. Thus \[U^*(X+\alpha, Y+\beta)=\frac{U(X+\alpha, Y+\beta)}{\bar{\beta}X -\bar{\alpha} Y}\] and
	\[S^*(X+\alpha, Y+\beta)=\frac{S(X+\alpha, Y+\beta)}{\bar{\beta}X -\bar{\alpha} Y}\]
	are both polynomials.
	
	Let $\cU^*$ be the curve defined by the polynomial $U^*(X+\alpha, Y+\beta)$. Assume that $\cU^*$ splits into two components $\cX$ and $\cY$. It is clear that $\bar{\beta}X -\bar{\alpha} Y$ does not divide $\frac{XY(X^{q-1}-Y^{q-1})}{\bar{\beta}X -\bar{\alpha}Y}$ which is the term of the second lowest degree of $U^*(X+\alpha, Y+\beta)$. By Lemma \ref{le:intersection_number_m_m1_coprime}, the intersection number $I((0,0),\cX\cap\cY)$ is zero.
	
	As $U^*(X+\alpha, Y+\beta) = V(X+\alpha,Y+\beta)S^*(X+\alpha, Y+\beta)$, we also get
	\[I((\alpha, \beta), \cA\cap \cB) \leq I((0,0),\cX\cap\cY)=0.\]
	
	Next we investigate the singular points of $\cV$ at infinity. By \eqref{eq:V^*} the points of $\cV$ at infinity are $R_\gamma$ for $\gamma \in \F_{q^2}\setminus \F_q$. To determine the intersection number of $\cA$ and $\cB$ at each $R_\gamma$, we consider
	\begin{align*}
	-H(Y,X+\gamma,1)/Y=&(X^{q^3}+\gamma^q - (X^q+\gamma^q) Y^{q^3-q} )(1-Y^{q-1})\\
	&- (X^q+\gamma^q-(X+\gamma)Y^{q-1})(1-Y^{q^3-q})\\
	=&(X^q-X+\gamma^q-\gamma)Y^{q^3-1} - X^{q^3}Y^{q-1}  + X^{q^3}\\
	&+XY^{q-1} - X^q+(\gamma - \gamma^q)Y^{q-1}.
	\end{align*}
	As $(\gamma - \gamma^q)Y^{q-1}$ is the term of the lowest degree in it, each point $R_\gamma$ is a non-ordinary singular point of $\cV$ of multiplicity $q-1$. Note that $Y\nmid XY^{q-1} - X^q$. By Lemma \ref{le:intersection_number_m_m1_coprime}, $I(R_\gamma, \cA \cap \cB)=0$.
	
	Denote the degrees of $\cA$ and $\cB$ by $d_1$ and $d_2$, respectively. By B\'ezout's theorem,
	\[ d_1 d_2 = \sum I(P, \cA\cap \cB).\]
	According to our previous calculation, the right-hand side of it equals $0$, whence one of $d_1$ and $d_2$ has to be $0$. Therefore, $\cV$ is absolutely irreducible and this completes the proof.
\end{proof}

\subsection{Further results and open problems}
We investigate further the curves of the previous part in order to show that $\cC_n$ is an MRD code if and only if a certain rank-metric code of dimension $2n$ over $\F_q$ in  $\F_q^{(n-1)\times n}$ is an MRD code.
Assume that ${H(1,X,Y)}/{W(1,X,Y)}=0$. By \eqref{eq:curve_F},
\[(Y^q-Y)^{q^2-1}+ (Y^q-Y)^{q-1}= (X^q-X)^{q^2-1} + (X^q-X)^{q-1}.  \]
If we set
\begin{equation}
\label{bar}
\overline{X} = X^q-X \mbox{ and } \overline{Y}=Y^q-Y,
\end{equation}
then it becomes
\begin{equation}\label{eq:Y1Y2}
\overline{Y}^{q^2-1} + \overline{Y}^{q-1}= \overline{X}^{q^2-1} + \overline{X}^{q-1}.
\end{equation}

Hence, $\cH$ contains no $\F_{q^n}$-rational points besides those on $\cW$ if and only if
every $\F_{q^n}$-rational point $\la (1,x,y)\ra$ on the curve defined by the affine equation \eqref{eq:Y1Y2} satisfies $(x^q-x)^{q-1}=(y^q-y)^{q-1}$.

Assume that $\overline{Y}^{q^2-1} + \overline{Y}^{q-1}= \overline{X}^{q^2-1} + \overline{X}^{q-1}=-t$, for some $t\in \F_{q^n}$. It follows that $\overline{X}$ and $\overline{Y}$ are both roots of
\begin{equation}\label{eq:Y}
Z^{q^2} + Z^{q} + tZ\in \F_{q^n}[Z].
\end{equation}
The polynomial \eqref{eq:Y} has at most $q^2$ roots.
If \eqref{eq:Y} has $q$ roots, then for any two non-zero roots, $z_1$ and $z_2$, it holds that $z_1^{q-1}= z_2^{q-1}$. This implies that the corresponding point $\la(1,x,y)\ra$ is on $\cW$.
If the polynomial \eqref{eq:Y} has $q^2$ roots, then there always exist two of them, $z_1$ and $z_2$, which are $\F_q$-linearly independent, or equivalently $z_1^{q-1}\neq z_2^{q-1}$. By \eqref{bar} the roots of \eqref{eq:Y} have to be in $\{x \colon \Tr_{q^n/q}(x)=0\}$.
Hence, $z_1^{q-1}= z_2^{q-1}$ holds for any two roots $z_1$ and $z_2$ of \eqref{eq:Y} if and only if \eqref{eq:Y} has at most $q$ roots in $\F_{q^n}$ with trace zero over $\F_q$.

Therefore, we have proved the following result.
\begin{proposition}\label{prop:t_eq}
	The set of linearized polynomials $\cC_n$ is an MRD code if and only if \eqref{eq:Y}
	has at most $q$ roots in $\{y\in \F_{q^n} \colon \Tr_{q^n/q}(y)=0\}$ for each $t\in \F_{q^n}$.
\end{proposition}

Proposition \ref{prop:t_eq} shows us that $\cC_n$ is an MRD if and only if  $\{aX+ bX^q+bX^{q^2} \colon a,b\in \F_{q^n}  \}$ with restriction to the $(n-1)$-dimensional $\F_q$-subspace $\{x\in \F_{q^n} \colon \Tr_{q^n/q}(x)=0\}$ of $\F_{q^n}$ is an MRD code of size $q^{2n}$ in $\F_{q}^{(n-1)\times n}$.


Besides the adjoint and Delsarte dual operations there is another operation which preserve the maximality of the idealisers of certain families of MRD codes and can be used to produce possibly new families:

\begin{proposition}\label{propos}
	Fix a prime power $q$ and an integer $n$.
	The set
	\begin{equation}
	\label{oldcode}
	\{a_0 X^{\sigma^{t_0}}+a_1 X^{\sigma^{t_1}}+\ldots+a_{k-1}X^{\sigma^{t_{k-1}}}\colon a_0,a_1,\ldots,a_{k-1}\in \F_{q^{nm}}\}
	\end{equation}
	with $\sigma=q^{ms}$, $\gcd(s,n)=1$ is an $\F_{q^m}$-linear MRD code for every positive integer $m$  if and only if
	\begin{equation}
	\label{newcode}
	\{a_0 X^{\tau^{t_0}}+a_1 X^{\tau^{t_1}}+\ldots+a_{k-1}X^{\tau^{t_{k-1}}}\colon a_0,a_1,\ldots,a_{k-1}\in \F_{q^{nm}}\}
	\end{equation}
	with $\tau=q^{mt}$, $\gcd(t,n)=1$ is an $\F_{q^m}$-linear MRD code for every positive integer $m$.
\end{proposition}
\begin{proof}
	Suppose that the condition holds for the codes defined by \eqref{oldcode}.
	Let $z$ denote the multiplicative inverse of $s$ modulo $n$, let $m$ and $t$ be any positive integers with $\gcd(t,n)=1$. By our assumption
	\begin{equation}
	\label{step1}
	\{a_0 X^{\sigma^{t_0}}+a_1 X^{\sigma^{t_1}}+\ldots+a_{k-1}X^{\sigma^{t_{k-1}}}\colon a_0,a_1,\ldots,a_{k-1}\in \F_{q^{nztm}}\}
	\end{equation}
	with $\sigma=q^{sztm}$, is an $\F_{q^{ztm}}$-linear MRD code.
	Equivalently, the elements of \eqref{step1} have kernels of dimension at most $k-1$ over $\F_{q^{ztm}}$. Let $U$ be the $\F_{q^{ztm}}$-subspace of the roots in  $\F_{q^{nztm}}$ of
	$f(X):=a_0 X^{\sigma^{t_0}}+a_1 X^{\sigma^{t_1}}+\ldots+a_{k-1}X^{\sigma^{t_{k-1}}}$ for some $a_0,a_1,\ldots, a_{k-1}\in \F_{q^{nm}}$. The polynomial $f$ is in \eqref{step1}, thus the dimension of $U$ over $\F_{q^{ztm}}$ is at most $k-1$.
	The $\F_{q^m}$-subspace of the roots of $f$ in the field $\F_{q^{nm}}$ is $U \cap \F_{q^{nm}}$. Since $\gcd(zt,n)=1$, according to \cite[Lemma 3.1]{lunardon_generalized_2015},
	the dimension over $\F_{q^m}$ of $U \cap \F_{q^{nm}}$ is at most $k-1$ and hence
	\[
	\{a_0 X^{\tau^{t_0}}+a_1 X^{\tau^{t_1}}+\ldots+a_{k-1}X^{\tau^{t_{k-1}}}\colon a_0,a_1,\ldots,a_{k-1}\in \F_{q^{nm}}\}\]
	with $\tau=q^{sztm}$ is an $\F_{q^m}$-linear MRD code. Since $sz \equiv 1 \pmod n$ this is the same code as the one defined in \eqref{newcode}.
\end{proof}

Now, assume that the monomials
$X^{\sigma^{t_0}},X^{\sigma^{t_1}},\ldots,X^{\sigma^{t_{k-1}}}$,
where $\sigma=q^{s}$ and $\gcd(s,n)=1$,
define an MRD code $\mathcal{C}_{\sigma,t_0,\ldots,t_{k-1}}$ over every extension $\mathbb{F}_{q^{mn}}$ of $\mathbb{F}_{q^{n}}$.
Then  Proposition \ref{propos} guarantees that $\mathcal{C}_{\tau,t_0,\ldots,t_{k-1}}$ is an MRD code over $\mathbb{F}_{q^{mn}}$ as well, with $\tau=q^{t}$ for any positive integer $t$ such that $\gcd(t,n)=1$.

It is easy to see that for generalized Gabidulin codes and for $\cC_7$ and $\cC_8$, the condition in Proposition \ref{propos} holds. If we apply Proposition \ref{propos} to $\cC_7$ or $\cC_8$, the resulting codes can be obtained also via the adjoint operation.

\begin{question}
	Is there any family of $n\times n$ MRD codes with maximum idealisers such that the condition in Proposition \ref{propos} does not hold?
\end{question}

To conclude our paper, we would like to present two open questions concerning the asymptotic behavior of MRD codes with maximum idealisers.
\begin{question}\label{ques:bence}
Is it true that for any positive integer $k$, there exists a constant $c$, depending only on $k$, such that for $n>c$ the set of linearized polynomials
	\[\left\{\sum_{i=0}^{k-1} a_i x^{q^{t_i}} \colon a_0,\ldots,a_{k-1} \in \F_{q^n}\right\}\]
	is an MRD code only if $t_0,\ldots,t_{k-1}$ is an arithmetic progression of $\mathbb{Z}_n$?
\end{question}

If the answer to Question \ref{ques:bence} is negative, then it is natural to ask the following, weaker question.
\begin{question}\label{ques:yue}
	Is it true that, for any $k$ distinct positive integers $t_0,t_1,\ldots, t_{k-1}$ which do not form an arithmetic progression of $\mathbb{Z}$, there exists a constant $c$ such that for $n>c$ the set of linearized polynomials
	\[\left\{\sum_{i=0}^{k-1} a_i x^{q^{t_i}} \colon a_0,\ldots,a_{k-1} \in \F_{q^n}\right\}\]
	is not MRD?
\end{question}

\begin{remark}
Recently, in \cite{BZ2} Bartoli, jointly with the fourth author of the paper,  analyzed Questions 4.6 and 4.7. In particular, they provide an affirmative answer to Question 4.7 for $q>5$, and a partial result for $q=2,3,4$ and $5$.
These results also yield classification of some special type of linear sets in \cite{NPZZ}.
\end{remark}

%

\end{document}